\documentclass[runningheads,envcountsame]{myllncs}
\usepackage{graphicx}
\usepackage{pgf,tikz,pgfplots,pgfkeys}
\usepackage{frenchineq}
\usepackage{subcaption}
\usepackage{mathtools}
\usepackage{bbm}
\DeclareMathAlphabet{\mathbx}{U}{BOONDOX-ds}{m}{n}
\SetMathAlphabet{\mathbx}{bold}{U}{BOONDOX-ds}{b}{n}
\DeclareMathAlphabet{\mathbbx} {U}{BOONDOX-ds}{b}{n}

\DeclarePairedDelimiterX{\dotp}[2]{\langle}{\rangle}{#1, #2} 

\usepackage{amsmath,mathrsfs,amssymb,amsthm}
\usepackage{mathtools}
\usepackage[shortlabels]{enumitem}
\usepackage{colortbl}
\usepackage{multirow} 
\usepackage{makecell} 
\usepackage{tikz-3dplot}

\usepackage[colorinlistoftodos,bordercolor=orange,backgroundcolor=orange!20,linecolor=orange,textsize=scriptsize]{todonotes}
\usetikzlibrary{calc}
\usetikzlibrary{shapes.geometric}
\usetikzlibrary{automata, positioning}
\usetikzlibrary{arrows}
\usetikzlibrary{arrows.meta}
\usepackage{hyperref}
\usepackage{xcolor}

\newcommand\myshade{85}
\definecolor{mylinkcolor}{HTML}{3f87d3}
\definecolor{mycitecolor}{HTML}{ee7c27}
\definecolor{myurlcolor}{rgb}{1,0.5,0}
\hypersetup{
  linkcolor  = mylinkcolor!\myshade!black,
  citecolor  = mycitecolor!\myshade!black,
  urlcolor   = myurlcolor!\myshade!black,
  colorlinks = true,
}
\providecommand{\doi}[1]{\url{https://doi.org/{#1}}}

\usepackage{cleveref}
\newcounter{tikzfigures}

\newcommand{\ie}{i.e.}

\newcommand{\R}{\mathbb{R}}
\newcommand{\Rp}{\R_{\geq 0}}
\newcommand{\onen}{[n]}

\newcommand{\cone}{{\mathrm{cone}}}
\newcommand{\conv}{{\mathrm{conv}}}
\newcommand{\Scal}{\mathcal{S}}
\newcommand{\Ccal}{\mathcal{C}}

\DeclareMathOperator{\arcenciel}{\mathrm{colint}}

\newcommand{\mediumcup}[1]{{\textstyle \bigcup\limits_{#1}}}
\newcommand{\mediumcap}[1]{{\textstyle \bigcap\limits_{#1}}}

\usepackage{wrapfig}
\usepackage[font=footnotesize]{caption}
\usepackage{ragged2e}
\usepackage{array}
\newcolumntype{L}[1]{>{\raggedright\let\newline\\\arraybackslash\hspace{0pt}}m{#1}}
\newcolumntype{C}[1]{>{\centering\let\newline\\\arraybackslash\hspace{0pt}}m{#1}}
\newcolumntype{R}[1]{>{\raggedleft\let\newline\\\arraybackslash\hspace{0pt}}m{#1}}

\spnewtheorem{definition}{Definition}{\bfseries}{\rmfamily}

\DeclareMathOperator{\interior}{\operatorname{int}}
\newcommand{\Rn}{\mathbb{R}^n}

\newcommand{\trop}{\textrm{trop}}
\let\emptyset\varnothing
\newcommand{\mapping}{P}
\newcommand{\LP}{\mathrm{LP}}

\newcommand{\reviewone}[1]{}
\newcommand{\reviewtwo}[1]{}

\newcommand{\Sbf}{\mathbf{S}}
\begin{document}
\title{A convex programming approach\\ to solve posynomial systems}
\titlerunning{A convex programming approach to solve posynomial systems}

\author{Marianne Akian \and Xavier Allamigeon \and Marin Boyet \and St\'ephane Gaubert}
\tocauthor{Marianne Akian, Xavier Allamigeon, Marin Boyet and St\'ephane Gaubert}
\institute{INRIA and CMAP, \'Ecole polytechnique, IP Paris, CNRS\\ \email{firstname.lastname@inria.fr}\\
}

\maketitle              
\begin{abstract}
  We exhibit a class of classical or tropical posynomial systems
  which can be solved by reduction to linear or convex programming problems.
  This relies on a notion of colorful vectors with respect to a collection
  of Newton polytopes. This extends the convex programming approach
  of one player stochastic games. 
\end{abstract}

\def\additionalproofs{1}

\section{Introduction}

A \emph{posynomial} is a function of the form
\[ P (x) = \sum_{a\in A} c_a x_1^{a_1}x_2^{a_2}\cdots x_n^{a_n} \]
where the variable
$x=(x_1,\dots,x_n)$
is a vector with real positive entries,
$A$ is a finite subset of vectors of $\Rn$, and
    the $c_a$ are positive real numbers.
    Here for any $a\in \Rn$, we denote by $a_i$ the $i$-th coordinate of $a$.
     The set $A$ is called the \emph{support} of $P$, also denoted by $S_P$, its elements are called the \emph{exponents} of the posynomial and the $c_a$ its \emph{coefficients}.

    Unlike polynomials, posynomials can have arbitrary exponents.
    They arise in convex optimization, especially in geometric and entropic programming~\cite{chandrasekaran} and in polynomial optimization~\cite{positivstellensatz}. They also arise in the theory of nonnegative tensors~\cite{lim05,shmuel}, in risk sensitive control~\cite{anantharam} and game theory~\cite{egames}.

A \emph{tropical posynomial} is a function of the form
\[ P^{\trop}(x) = \max_{a\in A}\left(c_a + \dotp{a}{x}\right)\]
where $\dotp{\cdot}{\cdot}$ is the usual dot product of $\Rn$,
the $c_a$ are now real coefficients, and $x=(x_1,\dots,x_n)$ can take
its values in $\R^n$.
The terminology used comes from the \emph{tropical} (or max-plus) \emph{semi-field}, whose additive law is the maximum and the multiplicative law is the usual sum.

In this paper, we are interested in solving (square) classical posynomial systems, that are of the form
    \begin{align}
      P_i(x) &= 1 \quad \text{for all} \; i \in \onen\coloneqq\{1,\dots,n\}\label{eq:classical}
    \end{align}
    with $x\in (\R_{>0})^n$, and the $P_i$ are classical posynomials.
    We will also study the tropical counterpart, 
    \begin{align}
      P^{\trop}_{i}(x) &= 0 \quad \text{for all} \; i \in \onen\label{eq:tropical}
    \end{align}
    with now $x\in \R^n$, and the $P_i^{\trop}$ are tropical posynomials (hereafter we shall write $P_i$ instead of $P_i^{\trop}$, for brevity).
    The optimality equations of Markov decision processes~\cite{puterman2014markov} are special
    cases of tropical posynomial systems.
    More general tropical posynomial systems
    arise in the performance analysis of timed
discrete event systems, see~\cite{emergency15}. 

Solving (square) posynomial systems is
in general NP-hard (\Cref{sec:complexity}).
However, we identify a tractable subclass.
The tropical
version can be solved exactly in polynomial time by reduction to a linear program (\Cref{sec:solveTropPosyn}), whereas the classical version can be solved approximatedly by reduction
to a geometric program (\Cref{sec:realPosy}). Our approach is based on a notion of colorful
interior of a collection of cones. A point is in the colorful interior
if it is a positive linear combination of vectors of these cones, and
if at least one vector of every cone is needed in such a linear combination.
Our reductions are valid when the colorful interior of the cone generated
by the supports of the posynomials is nonempty, and when a point in this
interior is known. As special cases, we recover the linear
programming formulation of Markov decision processes, and the geometric
programming formulation of risk sensitive problems. Properties of the colorful interior and related open problems are discussed in Section~\ref{sec:colorful}.

\section{Solving posynomial systems is NP-hard}
\label{sec:complexity}

The following two results show that the feasibility problems for classical or tropical posynomial systems are NP-hard, even with integer exponents.

\begin{proposition}
  \label{prop:nphard}
  Solving a square tropical posynomial system is NP-hard.
\end{proposition}
\begin{proof}
  We reduce \texttt{3-SAT} to the problem~\eqref{eq:tropical}. Let us consider a Boolean formula in conjunctive normal form $C_1\land \dots \land C_p$ made of $p$ clauses, each one of them using three out of $n$ real variables $x_1,\dots,x_n$ ($p,n\in\mathbb{N}$).

We introduce the following tropical posynomial system in the $2n+2p$ variables $(x_1,\dots,x_n,y_1,\dots,y_n,z_1\dots,z_p,s_1,\dots,s_p)$, with the same number of equations:
\vskip-1.7ex

\begin{align*}
   \forall i\in\onen \;\; & \max(x_i-1,y_i-1)=0 \,, \!\!&& x_i+y_i-1=0\,,  \label{e-neg}\\[1ex]
   \forall j\in[p]   \;\; & \max\Big( \underset{x_i\in C_j}{\max}(x_i-z_j), \!\!\underset{\neg x_i\in C_j}{\max}(y_i-z_j)\Big) = 0\,, \!\!&&\max(\tfrac{1}{2}-z_j,s_j-z_j)=0\, .
\end{align*}
This system can be constructed in polynomial time from the Boolean formula. The first $2n$ equations ensure that for all $i\in\onen$, $x_i\in\{0,1\}$ and that $x_i$ and $y_i$ have opposite logical values. The next $p$ equations express that for all $j\in[p]$, the variable $z_j$ has the same Boolean value as the clause $C_j$, with the notation $x_i\in C_j$ (resp.\ $\neg x_i\in C_j$)
if the variable $x_i$ occurs positively (resp.\ negatively) in the clause $C_j$. The last equations ensure that $z_j = 1$ for all $j \in [p]$.
The instance $C_1 \land \dots \land C_p$ is satisfiable if and only if this system admits a solution.
\end{proof}

\begin{theorem}
  Solving a square classical posynomial system is NP-hard.
  \end{theorem}
\begin{proof}
  We modify the previous construction to
  obtain a square posynomial system over $\R_{> 0}^{2n+2p}$,
  along the lines of Maslov's dequantization principle~\cite{litvinov2007maslov} or Viro's method~\cite{viro}: 
  \begin{align*}
    \forall i\in\onen \quad & \tfrac{2}{5}x_i + \tfrac{2}{5}y_i = 1 \,, \!\!&&\;\; x_i y_i =1\,,\\
    \forall j\in[p]   \quad & \sum_{x_i\in C_j}\tfrac{1}{6}x_i z_j^{-1} + \!\sum_{\neg x_i\in C_j}\tfrac{1}{6}y_i z_j^{-1} = 1 \,, \!\!&&\;\; \tfrac{1}{3}z_j^{-1}+s_jz_j^{-1}=1 \,.
  \end{align*}
  
From the first $2n$ equations, the variables $x_i$ and $y_i$ range over $\{2,1/2\}$, the values $2$ and $1/2$ respectively encode the true and false Boolean values. The variable $y_i = 1/x_i$ corresponds to the Boolean negation of $x_i$. 
Since each clause has precisely three literals,
using the $p$ next equations, we deduce that
the variable $z_j$ takes one of the values $\{1/2,3/4,1\}$
if the clause $C_j$ is satisfied, and that it takes the value $1/4$ otherwise.
The last $p$ equations impose that $z_j$ can take any value in $(1/3,\infty)$.
We deduce that the formula $C_1 \land \dots \land C_p$ is satisfied if and only if the posynomial system that we have obtained in this way admits a solution in $\R_{> 0}^{2n+2p}$.
\end{proof}

\section{A linear programming approach to solve tropical posynomial systems}\label{sec:solveTropPosyn}

Given tropical posynomials $P_1, \dots, P_n$, we write the system~\eqref{eq:tropical} as $P(x) = 0$, where $P \coloneqq (P_1, \dots, P_n)$. The \emph{support} of this system, denoted $\Sbf$, is defined as the disjoint union $\biguplus_{i \in \onen} S_{P_i}$ of the supports of the posynomials $P_i$. By {\em disjoint union}, we mean the coproduct in the category of sets (these supports may have non-empty intersections, and they may even coincide).

\begin{definition}\label{def:colorfulVector}
We say that a vector $y$ in the (convex) conic hull $\cone(\Sbf)$ is \emph{colorful} if, for all $\mu \in (\Rp)^\Sbf$,
\[
y=\sum_{a \in \Sbf} \mu_a\, a \implies  \forall i\in \onen \, , \; \exists a\in S_{P_i} \, , \; \mu_a > 0 \, .
\]
\end{definition}
In other words, a vector $y\in\Rn$ is colorful if  it arises as a nonnegative combination of the  exponents of $P$, but also if all such combinations make use of at least one exponent of each of the tropical posynomials $P_1,\dots,P_n$.

In this way, if we think of $S_{P_1},\dots,S_{P_n}$ as colored sets, we need all the colors to decompose a colorful vector $y$ over these. Moreover, by Carath\'eodory's theorem, every vector in the conic hull $\cone(\Sbf)$ can be written as a positive linear combination of an independent family of vectors of $\Sbf$. Hence, when $y$ is a colorful vector, it is obtained as a positive linear combination of precisely one vector $a_i$ in each color class $S_{P_i}$, and the family $a_1,\dots,a_n$ must be a basis. (If not, Carath\'eodory's Theorem would imply that $y$ is a positive linear combination of a proper subset of $\{a_1, \dots, a_n\}$, so that $y$ could not be a colorful vector.)

Given a vector $y$, we consider the following linear program:
\begin{equation}
\textrm{Maximize} \quad \dotp{y}{x} \quad \textrm{subject to} \quad \forall a \in \Sbf \, , \; c_a + \dotp{a}{x} \leq 0 \, .   
\tag{$\LP(y)$} \label{LP}
\end{equation}
Remark that the feasibility set of this linear program consists of the vectors $x \in \Rn$ satisfying $P(x) \leq 0$. In other words, it can be thought of as a relaxation of the system $P(x) = 0$. 
The following theorem shows that this relaxation provides a solution of $P(x) = 0$ if $y$ is a colorful vector.
\begin{theorem}\label{thm:LP}
  Assume that $y$ is a colorful vector, and
  that the linear program \labelcref{LP} is feasible.
  Then, the linear program \labelcref{LP}  has
  an optimal solution, and any optimal solution $x$ satisfies $P(x) = 0$.
\end{theorem}
  
\begin{proof} 
  Since the feasibility set of~\labelcref{LP},
  $\mathcal{F}\coloneqq \{x \in \Rn \colon P(x) \leq 0 \}$, is nonempty, we can consider its recession cone, which is given by $\mathcal{C}=\{x\in\mathbb{R}^n\;\colon\;\forall a \in\Sbf,\;\dotp{a}{x} \leq 0\}$. As a colorful vector, $y$ belongs to the polyhedral cone generated by the vectors $a \in \Sbf$, so $\dotp{y}{x} \leq 0$ for all $x \in \mathcal{C}$. By the Minkowski--Weyl theorem,
$\mathcal{F}$ is a Minkowski sum of the form $\mathcal{P}+\mathcal{C}$ where $\mathcal{P}$ is a polytope, i.e., every feasible point $x$ can be written as $x=x'+x''$ with $x'\in\mathcal{P}$ and $x''\in\mathcal{C}$. Since $\dotp{y}{x''}\leq 0$, the maximum of the objective function $x \mapsto \dotp{y}{x}$ over the polyhedron $\mathcal{F}$ is attained (by an element of $\mathcal{P}$).

Let $x^{\star}\in\mathbb{R}^n$ be an optimal solution of~\labelcref{LP}.
From the strong duality theorem, 
the dual linear program admits an optimal solution $(\mu^{\star}_a)_{a\in\Sbf} \in (\R_{\geq 0})^\Sbf$ which satisfies $y = \sum_{a\in\Sbf}\mu^{\star}_a\,a$ and $\mu^{\star}_a(c_a+\dotp{a}{x^{\star}})=0$ for all $a \in \Sbf$. Since $y$ is a colorful vector, for all $i\in\onen$, there is some $a_i\in S_{P_i}$ such that $\mu_{a_i}^{\star}>0$. We then get that, for all $i \in \onen$, $\mapping_i(x^{\star})\geq c_{a_i}+\dotp{a_i}{x^{\star}} = 0$. As a result, $\mapping(x^{\star})=0$.
\end{proof}

  We next provide a geometric condition ensuring that the linear 
  program \labelcref{LP} is feasible regardless of the coefficients $c_a$.
    We say that the tropical posynomial function $P$ has \emph{pointed} exponents if its support is contained in an open halfspace, \ie~there exists $z \in \Rn$ such that $\forall a\in \Sbf $, $\dotp{a}{z} < 0$.
Our interest for pointed systems comes from the following property:
\begin{proposition}
          \label{pointedToFeasible}
          The inequality problem $\mapping(x)\leq 0$ has a solution $x\in \R^n$
          regardless of the coefficients of 
          $P$ if and only if $\mapping$ has pointed exponents.
        \end{proposition}

        \begin{proof}
      Suppose that for all values of $(c_a)_{a\in\Sbf}$, there exists $x\in\Rn$ such that $\mapping(x)\leq 0$. By choosing $c_a\equiv1$, there exists $x_0\in\Rn$ that satisfies
      $\forall a\in \Sbf$, $1 + \dotp{a}{x_0} \leq  0$.
      Hence, 
      for all $i\in\onen$,
      the exponents of $P_i$ lie in the open halfspace $\{a\in\Rn\;|\; \langle a,x_0\rangle < 0\}$.

      Suppose now that $\mapping$ has pointed exponents. Then there is some $z\in\Rn$ such that for all $a\in \Sbf$, we have $\dotp{a}{z} < 0$. We define
      $\lambda \coloneqq \max_{a\in\Sbf}\;(-c_a)/\dotp{a}{z}$ 
      so that $\forall a\in \Sbf$, $c_a+\dotp{a}{\lambda z}\leq 0$ and therefore for all $i\in\onen$, $\mapping_i(\lambda z)\leq 0$.
    \end{proof}

As a consequence of~\Cref{thm:LP} and~\Cref{pointedToFeasible}, if the tropical posynomial system $P(x) = 0$ has pointed exponents and there exists a colorful vector, then the system admits a solution which can be found by linear programming.

  A remarkable special case consists of Markov decision processes. In this framework,
  the set $\onen$ represents the state space, and at each state $i\in\onen$,
  a player has a finite set $B_i$ of available actions included in the $n$-dimensional simplex $\{p\in \Rp^n\colon \sum_{j=1}^{n} p_j \leq 1\}$.
  If $p\in B_i$, $p_j$ stands for the probability that the next state
  is $j$, given that the current state is $i$ and action $p$
  is chosen by the player, so the difference $1-\sum_{j=1}^{n}p_j$
  is the death probability in state $i$ when this action is picked.
  To each action $p$
  is attached a reward $c_p\in \R$. Given an initial
  state $i\in\onen$, one looks for the value $v_i\in \R$,
  which is defined as the maximum over all the strategies
  of the expectation of the sum of rewards up to the death
  time, we refer the reader to \cite{puterman2014markov} for background.
  The value vector $v=(v_i)_{i\in\onen}$ is solution of the
  tropical posynomial problem
  \[
  v_i = \max_{p\in B_i}\;( c_p + \dotp{p}{v}) ,\quad \forall i\in\onen \enspace. 
  \]
  This reduces to the form~\eqref{eq:tropical} with $S_{P_i}\coloneqq B_i - e_i$,
  where $e_i$ denotes the $i$-th element of the canonical basis of $\R^n$.
  We say that a Markov decision process is of {\em discounted type}
  if for every state $i\in\onen$ there is at least one action $p\in B_j$ such that
  $\sum_{j=1}^{n}p_j<1$.
  \begin{proposition}\label{prop:markovrainbow}
    If a Markov decision process is of discounted type, then any negative vector is colorful with respect to the associated
    posynomial system. 
    \end{proposition}
\if\additionalproofs1
    
\begin{proof}
  Let $x$ be a vector with negative entries. For each $i\in\onen$, let us select
  a vector $p^{i}\in B_i$ such that $\sum_{j=1}^np^{i}_j<1$. We claim
  that $x$ is in the cone spanned by the vectors $(p^{i}-e_i)_{i\in\onen}$. Indeed,
  the system $x=\sum_{j=1}^n\lambda_i (p^{i}-e_i)$ can be rewritten
  as $\lambda = -x+ M\lambda $, where $M$ is the matrix with columns
  $(p^{i})_{i\in \onen}$, and $\lambda=(\lambda_i)_{i\in\onen}$. Since every column of $M$ of this matrix
  has a sum strictly less than $1$, the map $\lambda \mapsto M\lambda$ is a contraction
  in the $\ell_1$ norm. It follows that $(I-M)^{-1}=\sum_{k=0}^{\infty}M^k$ is a nonnegative matrix. Hence, the vector $\lambda = -(I-M)^{-1}x$ is nonnegative, 
  showing that $x$ is in the cone spanned by the $(S_{P_i})_{i\in\onen}$.
  
  By contradiction, suppose that $x$ is not a colorful vector. Then, there exists  a vector $\mu$ with nonnegative entries such that
  \[ 
  x = \sum_{j=1}^n\sum_{p \in B_j}\mu_p\big(p-e_j\big) \quad \text{and}\quad\exists i\in \onen \, , \; \forall p\in B_i \, , \; \mu_p = 0 \,,
  \]
  meaning that the $i$-th color is not used in the decomposition. It follows that 
  \begin{align*}
  x_i = x\cdot e_i & = \sum_{j=1}^n\sum_{p \in B_j}\mu_p \left(p-e_j\right)\cdot e_i = \sum_{j=1}^n\sum_{p \in B_j}\mu_p \Bigg(\sum_{k = 1}^n p_{ k}\delta_{ik}-\delta_{ij}\Bigg)\\
  & = \sum_{j=1}^n\sum_{p \in B_j}\mu_p p_j - \sum_{p\in B_i}\mu_p  = \sum_{j=1}^n\sum_{p \in B_j}\mu_p p_j 
  \end{align*}
  By nonnegativity of the multipliers $(\mu_p)_{p\in B_j}$ and of the entries of the vectors $p\in B_j$, we end up with $x_i\geq 0$, hence the contradiction.
  \end{proof}
  \fi

  Thus, we recover the linear programming approach to Markov decision processes (see~\cite{puterman2014markov}), showing that the value is obtained by minimizing the function $v\mapsto \sum_{i\in[n]}v_i$ subject to the constraints $v_i \geq c_p + \dotp{p}{v}$ for $i\in[n]$ and $p\in B_i$.
  
\section{Geometric programming approach of posynomials systems}
  \label{sec:realPosy}

We refer the reader to~\cite{chandrasekaran} for background on geometric programming.

Given a collection $P = (P_1, \dots, P_n)$ of classical posynomials, we now deal with the system $P_i(x) = 1$ for all $i \in \onen$, which, for brevity, we denote by $P(x) = 1$. We keep the notation of~\Cref{sec:solveTropPosyn} for the supports of the posynomials. Moreover, the definitions of colorful vectors and pointed exponents, which only depend on these supports, still make sense in the setting of this section.

\begin{lemma}
  \label{sublevelsets}
  If $y$ is a colorful vector, the polyhedron $\mathcal{P}$ defined by
\[ 
\mathcal{P}\coloneqq\bigl\{x\in\mathbb{R}^n \colon \forall a\in \Sbf ,\quad \log c_a+\dotp{a}{x} \leq 0\quad \text{and}\quad \dotp{y}{x} \geq \mu\bigr\} \] 
is bounded (possibly empty), regardless of our choice of positive $(c_a)_{a\in\Sbf}$ or $\mu\in\R$.\end{lemma}
\begin{proof}
If $\mathcal{P}$ is nonempty, let $\mathcal{C}\coloneqq\{x \in \R^n \;|\; \forall a \in \Sbf, \; \dotp{a}{x}\leq 0 \, , \; \dotp{y}{x} \geq 0 \}$ denote its recession cone, and let $x \in \mathcal{C}$. Since $y$ is a colorful vector, there exists $(\lambda_1,\dots,\lambda_n)\in \R_{> 0}^n$ and a basis $(a_1,\dots,a_n)\in\prod_{i\in\onen} S_{P_i}$ such that $y=\sum_{i=1}^n \lambda_i a_i$. Thus, $\dotp{y}{x} \leq 0$, and so $\dotp{y}{x} = \sum_{i = 1}^n \lambda_i \dotp{a_i}{x} = 0$. As a consequence, since $\lambda_i > 0$ for all $i\in\onen$, $\dotp{a_i}{x} = 0$. Since  $(a_1,\dots,a_n)$ is a basis, we get $x = 0$. Thus, $\mathcal{C} = \{0\}$, and $\mathcal{P}$ is bounded by Minkowski--Weyl Theorem.
\end{proof}
  
Given $X \in \R^n$, we denote by $\exp X$ the vector with entries $\exp X_i$, $i \in \onen$.
\begin{theorem}
    \label{thm:entropic}
Let $\mapping(x) = 1$ be a posynomial system with pointed exponents, and $y$ be a colorful vector. Then, the system has a solution $x = \exp X^* \in (\R_{> 0})^n$,  where $X^*$ is an arbitrary solution of the following geometric program:
\begin{equation}
\textrm{Maximize} \quad {\dotp{y}{X}}\qquad \text{subject to}\qquad \forall i\in\onen \quad g_i(X)\leq 0\, ,
\label{pb-G}
\tag{G}
\end{equation}
where $g_i(X) \coloneqq \log\Big(\sum_{a\in S_{P_i}}c_{a}\,e^{\dotp{a}{X}}\Big)$.
\end{theorem}

\begin{proof}
  For $x\in \mathbb{R}^n_{>0}$, we define $X=\log(x)$ (component-wise) so that $\mapping(x)=1$ is equivalent to solving $g_i(X) = 0$ for all $i \in \onen$. By H\"older's inequality, the functions $(g_i)_{i\in\onen}$ are convex.   
  We define $h_i \colon X\mapsto \max_{a\in S_{P_i}}\big(\log(c_{a}) + \dotp{a}{X}\big)$ for $i\in\onen$ and we observe that $h_i(X)\leq g_i(X) \leq h_i(X)+\log(|S_{P_i}|)$.
  
Since the system $\mapping(x) = 1$ has pointed exponents, by~\Cref{pointedToFeasible}, the polyhedron $\{X \in \R^n \colon \forall i\in\onen \, , \; h_i(X)+\log(|S_{P_i}|)\leq 0\}$ is nonempty.
  A fortiori, the feasible set of~\eqref{pb-G} is nonempty.
  
  Let us now prove that the maximum of~\eqref{pb-G} is finite and attained, by proving that the $\mu$-superlevel set $\mathcal{S}_{\mu}=\{X \in \R^n \colon \dotp{y}{X}\geq \mu\; \text{and}\; \forall i\in\onen \, , \; g_i(X)\leq 0\}$ of the objective function (included in the feasible set) is compact for all $\mu\in\mathbb{R}$. Closedness is direct, and we observe that for $\mu\in\mathbb{R}$, $\mathcal{S}_{\mu}\subset \{X \in \R^n \colon \dotp{y}{X}\geq \mu\; \text{and}\; \forall i\in\onen \, , \;  h_i(X)\leq 0\}$, but by~\Cref{sublevelsets}, this polyhedron is bounded. Hence, \eqref{pb-G} admits an optimal solution $X^{\star}$.
  
  Furthermore, again by~\Cref{pointedToFeasible}, there exists $\overline{X}$ such that for all $i\in\onen$, $h_i(\overline{X})+\log(|S_{P_i}|)+1\leq 0$. Therefore, for all $i\in\onen$, $g_i(\overline{X})<0$, which means that \eqref{pb-G} satisfies Slater's condition. Problem~\eqref{pb-G} being convex, optimality of $X^{\star}$ is characterized by the Karush--Kuhn--Tucker conditions (see~\cite{boyd2004convex}). Hence, there is a vector of nonnegative multipliers $\lambda^{\star}=(\lambda^{\star}_1,\dots,\lambda^{\star}_n)$ such that $(X^{\star},\lambda^{\star})$ is a stationary point of the Lagrangian of~\eqref{pb-G}, and the complementarity slackness conditions hold, \ie~for all $i\in\onen$, $\lambda_i ^{\star}\,g_i(X^{\star})=0$.
  Defining $Z_i \coloneqq \sum_{a\in S_{P_i}}c_{a} e^{\dotp{a}{X^{\star}}}>0$
   for $i\in\onen$, the stationarity conditions give 
\[
y = \sum_{i=1}^n\frac{\lambda_i^{\star}}{Z_i}\sum_{a\in S_{P_i}}c_{a}\, e^{\dotp{a}{X^{\star}}}\,a \, . 
\]
Since $y$ is colorful, for all $i\in\onen$, $\lambda_i^{\star}>0$. The complementarity slackness conditions yield $g_i(X^{\star})=0$ for all $i\in\onen$.
  So $x^{\star}\coloneqq\exp(X^{\star})$ satisfies $\mapping(x^{\star})=1$.
    \end{proof}
  
  \section{Properties of the colorful interior of convex sets}
  \label{sec:colorful}

\Cref{thm:LP,thm:entropic} rely on the existence of a colorful vector. The purpose of this section is to study the properties of the set of such vectors. In fact, colorful vectors can be defined more generally from a family of $n$ closed convex cones.
\begin{definition}\label{def:colorfulInterior}
Let $\Ccal = (C_1,\dots,C_n)$ be a collection of $n$ closed convex cones of $\Rn$. A vector $y \in \Rn$ is said to be \emph{colorful} if it belongs to the set

\[
\cone (C_1 \cup \dots \cup C_n) 
\setminus \mediumcup{i\in\onen}\mathrm{cone} \big(\mediumcup{j\neq i} C_{j}\big) \, .
\]
The latter set is referred to as the \emph{colorful interior} of $\Ccal$.
\end{definition}
Remark that~\Cref{def:colorfulVector} can be recovered by taking $C_i \coloneqq \cone(S_{P_i})$ for all $i \in \onen$. 
In what follows, we restrict to the case where the collection $\Ccal$ 
is \emph{pointed}, \ie~$\cone(C_1 \cup \dots \cup C_n)$ is a pointed cone (in the non pointed case, the colorful interior enjoys much less structure than the one proved in~\Cref{th:colint}, in particular it may not even be connected). Suppose that $\{x\in\Rn \colon \dotp{z}{x}>0\}$ is an open halfspace containing the $(C_i)_{i\in\onen}$. Then, as a cone, the colorful interior of $\Ccal$ can be more simply studied from its cross-section with $\{x \in \Rn \colon \dotp{x}{z}=1\big\}$. The latter can be shown to coincide with the set 
\begin{equation}\label{eq:cross_section}
\conv (S_1 \cup \dots \cup S_n) 
\setminus \mediumcup{i\in\onen}\mathrm{conv} \big(\mediumcup{j\neq i} S_j\big)
\end{equation}
where for $i\in\onen$, $S_i$ is the cross-section of the cone $C_i$ by $\{x \in \Rn \colon \dotp{x}{z}=1\big\}$. Given a collection $\Scal = (S_1, \dots, S_n)$ of closed convex sets of $\R^{n-1}$, we refer to the set~\eqref{eq:cross_section} as the \emph{colorful interior} of $\Scal$, and denote it by $\arcenciel{\Scal}$. We start with a lemma justifying the terminology we have chosen:
\begin{lemma}\label{lemma:open}
Let $\Scal = (S_1, \dots, S_n)$ be a collection of $n$ closed convex sets of $\R^{n-1}$. Then $\arcenciel{\Scal}$ is an open set included in $\interior \conv(S_1\cup\dots\cup S_n)$.
\end{lemma}
\if\additionalproofs1

\begin{proof}
  Since $\arcenciel{\Scal}\subset\conv(\Scal)$,
  it suffices to prove that $\arcenciel{\Scal}\cap\partial\conv(\Scal)=\emptyset$
  to show $\arcenciel{\Scal}\subset\interior\conv(\Scal)$.
  Let $x\in\partial \conv(\Scal)$, and let $H$ be a supporting hyperplane of $\conv(\Scal)$ at $x$, so that $x\in\conv(\Scal) \cap H = \conv(\bigcup_{i\in\onen}(S_i\cap H))$. By Carath\'{e}odory's theorem, $x$ can be written using at most $n-1$ vectors of $(S_i\cap H)_{i\in\onen}$, so it is in $\conv(\bigcup_{j\neq i}S_i)$ for some $i\in\onen$ (the color not used in the decomposition). Therefore $x\not\in\arcenciel{\Scal}$.

  Now we show $\arcenciel{\Scal}$ is open. Let $i\in\onen$, $x$ is not in $\conv(\bigcup_{j\neq i}S_i)$ which is closed, so $d_i := \mathrm{d}(x,\conv(\bigcup_{j\neq i}S_i)) > 0$. Similarly, $\partial\conv(\Scal)$ is closed and does not contain $x$ by the above argument, so $d':=\mathrm{d}(x,\conv(\Scal))>0$. The open ball centered in $x$ with radius $\min((d_i)_{i\in\onen},d')$ is clearly included in $\conv(\Scal)$.
\end{proof}
\fi
The set $\arcenciel{\Scal}$ has appeared in a work of Lawrence and Soltan~\cite{La09}, in the proof of the characterization of the intersection of convex transversals to a collection of sets. In more details, \Cref{lemma:open} and~\cite[Lemma~6]{La09} imply:
\begin{proposition}\label{prop:transversal}
Let $\Scal = (S_1, \dots, S_n)$ be a collection of $n$ closed convex sets of $\R^{n-1}$. 
Define $\mathcal{D}\coloneqq\{\conv(\{x_1,\dots,x_n\}) \colon x_1\in S_1,\dots,x_n\in S_n\}$, the set of \emph{colorful simplices}, \ie~with one vertex in each colored set.
Then we have

\[ 
\arcenciel{\Scal} = \bigcap_{\Delta \in \mathcal{D}} \interior\Delta= \interior\bigcap_{\Delta\in\mathcal{D}} \Delta \enspace .
\]

\end{proposition}
Remark that~\Cref{prop:transversal} still holds if the colorful simplices $\Delta\in\mathcal{D}$ are replaced by the convex transversals to the sets $S_1, \dots, S_n$.

Given a hyperplane $H \coloneqq \{x\in\R^{n-1}\colon \dotp{h}{x} = b \}$, we shall denote below by $H^{>}$ (resp.\, $H^{\leq}$) the open (resp.\, closed) halfspace $\{x\in\R^{n-1}\colon \dotp{h}{x} > b\}$ (resp.\,  $\{x\in\R^{n-1}\colon \dotp{h}{x} \leq b\}$).
As a corollary of~\cite[Th.~2]{La09}, we get the following characterization of the colorful interior:

\begin{theorem}\label{th:colint}
Let $\Scal = (S_1,\dots,S_n)$ be a collection of $n$ closed convex sets of $\R^{n-1}$, and assume that $\arcenciel{\Scal}$ is nonempty. Then, $\arcenciel{\Scal}$ is the interior of a $(n-1)$-dimensional simplex. 

Moreover, if the sets $(S_i)_{i\in\onen}$ are bounded, then there are $n$ unique hyperplanes $(H_i)_{i\in\onen}$ such that for all $i\in\onen$, $S_i\subset H_i^>$, and for all $j\neq i$, $S_j\subset H_i^{\leq}$ and $S_j \cap H_i\neq \emptyset$. In this case, we have $\arcenciel{\Scal} = \bigcap_{i\in\onen} H_i^>$.

\end{theorem}
Geometrically, every $H_i$ in~\Cref{th:colint} is a tangent hyperplane to the convex sets $(S_j)_{j \neq i}$ which separates them from the set $S_i$. The existence (and uniqueness) of such tangent hyperplanes follows from the work of Cappell et al.~\cite{Ca94}, see also the work of Lewis, Klee and von Hohenbalken~\cite{Le96} for a constructive proof.
We depict on~\Cref{fig:colorful1} three colored sets $S_1, S_2$ and $S_3$ in $\R^2$ with nonempty colorful interior $\arcenciel{(S_1,S_2,S_3)}$, illustrating that the latter is the interior of a simplex as claimed in~\Cref{th:colint}. 

\definecolor{colorC}{HTML}{FF7F0E}
\definecolor{colorB}{HTML}{1F77B4}
\definecolor{colorA}{HTML}{2CA02C}

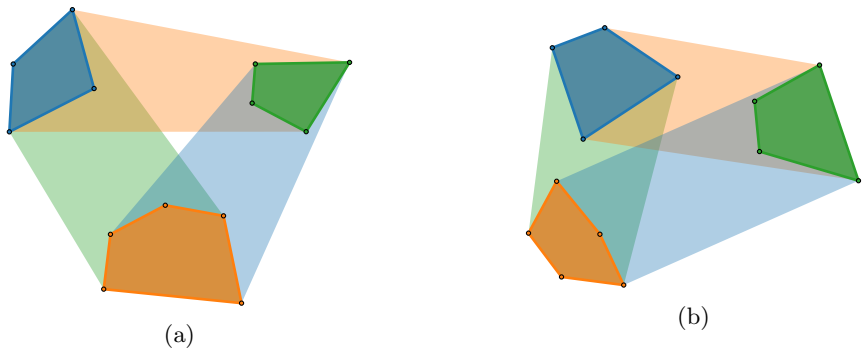
\begin{figure}[t]
  \centering
  \begin{subfigure}{0.44\textwidth}
  \def\tkzscl{.32}\centering
\definecolor{ududff}{rgb}{0.30196078431372547,0.30196078431372547,1}
\begin{tikzpicture}[line cap=round,line join=round,>=triangle 45,x=1cm,y=1cm,scale=\tkzscl]
\fill[line width=1pt,color=colorA,fill=colorA,fill opacity=0.36] (0.54,0.74) -- (3,3) -- (9.278,-5.572) -- (10.026,-9.202) -- (4.3,-8.62) -- (0.38,-2.08) -- cycle;
\fill[line width=1pt,color=colorC,fill=colorC,fill opacity=0.36] (0.38,-2.08) -- (12.71,-2.074) -- (14.514,0.808) -- (3,3) -- (0.54,0.74) -- cycle;
\fill[line width=1pt,color=colorB,fill=colorB,fill opacity=0.36] (4.3,-8.62) -- (4.58,-6.34) -- (10.598,0.742) -- (14.514,0.808) -- (10.026,-9.202) -- cycle;
\fill[line width=1pt,color=colorB,fill=colorB,fill opacity=0.7] (0.54,0.74) -- (0.38,-2.08) -- (3.9,-0.28) -- (3,3) -- cycle;
\fill[line width=1pt,color=colorA,fill=colorA,fill opacity=0.7] (10.598,0.742) -- (10.466,-0.886) -- (12.71,-2.074) -- (14.514,0.808) -- cycle;
\fill[line width=1pt,color=colorC,fill=colorC,fill opacity=0.7] (4.3,-8.62) -- (4.58,-6.34) -- (6.858,-5.132) -- (9.278,-5.572) -- (10.026,-9.202) -- cycle;
\draw [line width=1pt,color=colorB] (0.54,0.74)-- (0.38,-2.08);
\draw [line width=1pt,color=colorB] (0.38,-2.08)-- (3.9,-0.28);
\draw [line width=1pt,color=colorB] (3.9,-0.28)-- (3,3);
\draw [line width=1pt,color=colorB] (3,3)-- (0.54,0.74);
\draw [line width=1pt,color=colorA] (10.598,0.742)-- (10.466,-0.886);
\draw [line width=1pt,color=colorA] (10.466,-0.886)-- (12.71,-2.074);
\draw [line width=1pt,color=colorA] (12.71,-2.074)-- (14.514,0.808);
\draw [line width=1pt,color=colorA] (14.514,0.808)-- (10.598,0.742);
\draw [line width=1pt,color=colorC] (4.3,-8.62)-- (4.58,-6.34);
\draw [line width=1pt,color=colorC] (4.58,-6.34)-- (6.858,-5.132);
\draw [line width=1pt,color=colorC] (6.858,-5.132)-- (9.278,-5.572);
\draw [line width=1pt,color=colorC] (9.278,-5.572)-- (10.026,-9.202);
\draw [line width=1pt,color=colorC] (10.026,-9.202)-- (4.3,-8.62);
\begin{scriptsize}
\draw [fill=colorB] (0.54,0.74) circle (2.5pt);
\draw [fill=colorB] (0.38,-2.08) circle (2.5pt);
\draw [fill=colorB] (3.9,-0.28) circle (2.5pt);
\draw [fill=colorB] (3,3) circle (2.5pt);
\draw [fill=colorA] (10.598,0.742) circle (2.5pt);
\draw [fill=colorA] (10.466,-0.886) circle (2.5pt);
\draw [fill=colorA] (12.71,-2.074) circle (2.5pt);
\draw [fill=colorA] (14.514,0.808) circle (2.5pt);
\draw [fill=colorC] (4.3,-8.62) circle (2.5pt);
\draw [fill=colorC] (4.58,-6.34) circle (2.5pt);
\draw [fill=colorC] (6.858,-5.132) circle (2.5pt);
\draw [fill=colorC] (9.278,-5.572) circle (2.5pt);
\draw [fill=colorC] (10.026,-9.202) circle (2.5pt);
\end{scriptsize}
\end{tikzpicture}

  \caption{ }
  \label{fig:colorful1}
  \end{subfigure}
  \hspace*{\fill}
  \begin{subfigure}{0.44\textwidth}
  \def\tkzscl{.32}\centering
\definecolor{ududff}{rgb}{0.30196078431372547,0.30196078431372547,1}
\definecolor{cqcqcq}{rgb}{0.7529411764705882,0.7529411764705882,0.7529411764705882}
\begin{tikzpicture}[line cap=round,line join=round,>=triangle 45,x=1cm,y=1cm,scale=\tkzscl]
\fill[line width=1pt,color=colorA,fill=colorA,fill opacity=0.36] (5.715889278449259,-9.60027135021257) -- (7.963099852875052,-0.948510638673266) -- (4.9293655774002305,1.0964509840542056) -- (2.749571320207211,0.26498307151666217) -- (1.760798667459862,-7.442949198763808) -- (3.1315971178595956,-9.2631897640487) -- (5.715889278449259,-9.60027135021257) -- (5.715889278449259,-9.60027135021257);
\fill[line width=1pt,color=colorC,fill=colorC,fill opacity=0.36] (15.468783171457208,-5.263154941570789) -- (13.850791557870636,-0.45412431229959155) -- (4.9293655774002305,1.0964509840542056) -- (2.749571320207211,0.26498307151666217) -- (4.030481347629913,-3.5328027992629276) -- (15.468783171457208,-5.263154941570789) -- (15.468783171457208,-5.263154941570789);
\fill[line width=1pt,color=colorB,fill=colorB,fill opacity=0.36] (5.715889278449259,-9.60027135021257) -- (15.468783171457208,-5.263154941570789) -- (13.850791557870636,-0.45412431229959155) -- (2.9293481661612746,-5.285627047315046) -- (1.760798667459862,-7.442949198763808) -- (3.1315971178595956,-9.2631897640487) -- (5.715889278449259,-9.60027135021257) -- (5.715889278449259,-9.60027135021257);
\draw[line width=1pt,color=colorC,fill=colorC,fill opacity=0.7] (5.715889278449259,-9.60027135021257) -- (4.7271166257019095,-7.487893410252323) -- (2.9293481661612746,-5.285627047315046) -- (1.760798667459862,-7.442949198763808) -- (3.1315971178595956,-9.2631897640487) -- (5.715889278449259,-9.60027135021257) -- (5.715889278449259,-9.60027135021257);
\draw[line width=1pt,color=colorB,fill=colorB,fill opacity=0.7] (4.030481347629913,-3.5328027992629276) -- (7.963099852875052,-0.948510638673266) -- (4.9293655774002305,1.0964509840542056) -- (2.749571320207211,0.26498307151666217) -- (4.030481347629913,-3.5328027992629276) -- (4.030481347629913,-3.5328027992629276);
\draw[line width=1pt,color=colorA,fill=colorA,fill opacity=0.7] (15.468783171457208,-5.263154941570789) -- (13.850791557870636,-0.45412431229959155) -- (11.15413886855968,-1.959755397164873) -- (11.356387820258004,-4.04966123138086) -- (15.468783171457208,-5.263154941570789) -- (15.468783171457208,-5.263154941570789);
\begin{scriptsize}
    \draw [fill=colorB] (4.030481347629913,-3.5328027992629276) circle (2.5pt);
    \draw [fill=colorB] (4.9293655774002305,1.0964509840542056) circle (2.5pt);
    \draw [fill=colorA] (11.356387820258004,-4.04966123138086) circle (2.5pt);
    \draw [fill=colorB] (2.749571320207211,0.26498307151666217) circle (2.5pt);
    \draw [fill=colorB] (7.963099852875052,-0.948510638673266) circle (2.5pt);
    \draw [fill=colorA] (13.850791557870636,-0.45412431229959155) circle (2.5pt);
    \draw [fill=colorA] (11.15413886855968,-1.959755397164873) circle (2.5pt);
    \draw [fill=colorA] (15.468783171457208,-5.263154941570789) circle (2.5pt);
    \draw [fill=colorC] (3.1315971178595956,-9.2631897640487) circle (2.5pt);
    \draw [fill=colorC] (4.7271166257019095,-7.487893410252323) circle (2.5pt);
    \draw [fill=colorC] (5.715889278449259,-9.60027135021257) circle (2.5pt);
    \draw [fill=colorC] (2.9293481661612746,-5.285627047315046) circle (2.5pt);
    \draw [fill=colorC] (1.760798667459862,-7.442949198763808) circle (2.5pt);
\end{scriptsize}
\end{tikzpicture}

  \caption{ }
  \label{fig:colorful_counteraxample}
  \end{subfigure}
  \caption{
    (a) three convex sets $S_1$ (blue), $S_2$ (green) and $S_3$ (orange) in $\R^2$ and their colorful interior (white). The sets $(\widehat{S}_i)_{1\leq i\leq 3}$ (resp.\ $(\overline{S}_i)_{1\leq i\leq 3}$) are seen by taking convex hulls of $(S_i)_{1\leq i\leq n}$ (resp.\ intersection of $(\widehat{S}_i)_{1\leq i \leq 3}$) pairwise. Observe that the edges of the colorful interior are supported by tangent hyperplanes to two sets of $(S_1,S_2,S_3)$.\\
    (b) the colorful interior of $(S_1,S_2,S_3)$ is here empty, although these sets are separated (any three points in each of them are in general position), contrary to the sets $(\overline{S}_1,\overline{S}_2,\overline{S}_3)$, whose intersection is seen in the center of the figure.
  }
\end{figure}
  
Given a collection $\Scal = (S_1, \dots, S_n)$ of $n$ closed convex sets of $\R^{n-1}$, 
we now discuss necessary and sufficient conditions for $\arcenciel \Scal$ to be nonempty. To this purpose we recall that the collection $\Scal$ is \emph{separated} if for any choice of $k \leq n$ points $x_1, \dots, x_k$ in $S_{i_1}\times\dots\times S_{i_k}$ (where $i_1, \dots, i_k$ are pairwise distinct), the points $x_1,\dots,x_k$ are in general position (spanning a $(k-1)$-dimensional affine space).

\begin{proposition}\label{prop:equivHatBar}
  Let $S_1,\dots,S_n$ be a collection of $n$ compact convex sets of $\R^{n-1}$, and let us define $\widehat{S}_i\coloneqq \conv(\bigcup_{j\neq i}S_j)$ for all $i\in\onen$.\\
  Then, the family $(\overline{S}_i)_{i\in\onen}$ is separated if and only if $\bigcap_{i\in\onen} \widehat{S}_i = \emptyset$.
\end{proposition}
\if\additionalproofs1

\begin{proof}
  Recall that the collection $\Scal$ is separated if and only if for all $k+\ell\leq n$, any two disjoint $k$-subcollection and $\ell$-subcollection of $\mathcal{S}$ can be separated by an affine hyperplane (this result, already stated in~\cite{barany2008slicing} for example, is easy to show).

  We first suppose that $\bigcap_{i\in\onen}\widehat{S}_i=\emptyset$, then for all partition $(I,J)$ of $\onen$, we have $\big(\bigcap_{i\in I}\widehat{S}_i\big)\cap\big(\bigcap_{j\in J}\widehat{S}_j\big)=\emptyset$. The convex sets separation theorem ensures that there exists an affine hyperplane $H$ separating the convex sets $\bigcap_{i\in I}\widehat{S}_i$ and $\bigcap_{j\in J}\widehat{S}_j$. However, for all $j\in J$, $\overline{S}_{I}\coloneqq \conv\big((\overline{S}_{i})_{i\in I}\big)\subset \widehat{S}_j$, as a result $\overline{S}_{I}\subset\bigcap_{j\in J}\widehat{S}_j$. In particular, $H$ separates the two subcollections $(\overline{S}_{i})_{i\in I}$ and $(\overline{S}_{j})_{j\in J}$. If $(I',J')$ are two disjoint subcollections of $\onen$ such that $I'\cup J'\neq\onen$, we can still complete $(I',J')$ in a partition of the form $(I,J)$ to get as above a hyperplane $H$ separating  $(\overline{S}_{i})_{i\in I'}$ and $(\overline{S}_{j})_{j\in J'}$. From the proposition we have recalled on the separation property, the collection $(\overline{S}_{i})_{i\in\onen}$ is separated.

Conversely, the separation of $(\overline{S}_{i})_{i\in\onen}$ implies that for all $i\in\onen$, there exists some affine hyperplane separating  $\overline{S}_{i}$ from the collection $(\overline{S}_{k})_{k\neq i}$. By convexity, this hyperplanes separates as well $\overline{S}_{i}$ and $\conv \big(\bigcup_{k\neq i}\overline{S}_{k}\big)=\widehat{S}_i$. As a result $\overline{S}_{i}\cap\widehat{S}_i=\emptyset$, or equivalently $\bigcap_{i\in\onen}\widehat{S}_i=\emptyset$.
\end{proof}
\fi

\begin{proposition}\label{prop:barSeparated}
  Let $S_1,\dots,S_n$ be a collection of $n$ compact convex sets of $\R^{n-1}$. Let us define, for all $i\in\onen$, 
  \[ 
  \overline{S}_i \coloneqq \mediumcap{j\neq i} \conv \big( \mediumcup{k\neq j} S_k \big) \, . 
  \]    
  Then, if $\arcenciel{\Scal}$ is nonempty, the family $(\overline{S}_i)_{i\in\onen}$ is separated.
  \end{proposition}
\if\additionalproofs1
  
  \begin{proof}
    We still denote $\widehat{S}_i \coloneqq \conv(\bigcup_{j\neq i} S_j)$. We will show by contraposition that $\arcenciel{\Scal}\neq\emptyset\Longrightarrow \bigcap_{i\in\onen}\widehat{S}_i=\emptyset$. Thus, suppose we have $x_0$ such that for all $i\in\onen$, $x_0 \in \widehat{S}_i$.
    Let $x\in\conv(\Scal)$. If $x = x_0$, it is clear that $x\notin\arcenciel{\Scal}$.
    Otherwise, consider the halfline $[x_0, x)$, which is not entirely included in $\conv(\Scal)$ because all the $(S_i)_{i\in\onen}$ are bounded. We define $x'\coloneqq \max_{z\in[x_0,x)}\{z \in \conv(\Scal)\}$, in the sense of the order induced on $[x_0,x)$. Note that the maximum is indeed attained by closedness of the $(S_i)_{i\in\onen}$. By definition, we have $x\in[x_0,x']$ and $x'\in\partial\conv(\Scal)$. The latter ensures that $x\in\conv(\bigcup_{i\in\onen}(S_i\cap H))$ where $H$ is a supporting hyperplane to $\conv(\Scal)$. Hence, from Carathéodory's theorem, there is $k\in\onen$ such that $x'\in \widehat{S}_k$ (see the proof of~\Cref{lemma:open} where it is also done). But by assumption, $x_0$ is in $\widehat{S}_k$ as well and so does $x$ by convexity. We have thus just proved that for all $x\in\conv(\Scal)$, $x\in\widehat{S}_k$ for some $k\in\onen$, which implies that $\arcenciel{\Scal}=\emptyset$. \Cref{prop:equivHatBar} terminates the proof.
  \end{proof}
  \fi 
  
\Cref{prop:barSeparated} provides a necessary condition to ensure that $\arcenciel{\Scal}\neq\emptyset$. Since for all $i\in\onen$, we have $S_i\subset\overline{S}_i$, we also obtain that the separation of $(S_i)_{i\in\onen}$ is necessary as well for $\arcenciel{\Scal}$ to be nonempty. However, \Cref{fig:colorful_counteraxample} shows that this last condition is not sufficient. We conjecture that the necessary condition stated in \Cref{prop:barSeparated} is sufficient:
\begin{conjecture}\label{conj}
Let $S_1,\dots,S_n$ be a collection of $n$ compact convex sets of $\R^{n-1}$. Then $\arcenciel{\Scal}$ is nonempty if and only if the family $(\overline{S}_i)_{i\in\onen}$ is separated.
\end{conjecture}
We prove this conjecture in the case where $n = 3$ (it is also straightforward to establish for $n=2$).
\begin{proposition}
Let $\Scal = (S_1, S_2, S_3)$ be a collection of three convex compact sets of $\R^2$. Then, $\arcenciel{\Scal}$ is nonempty if and only if $(\overline{S}_1, \overline{S}_2, \overline{S}_3)$ is separated.  
\end{proposition}
\begin{proof}
  Suppose that $(\overline{S}_1, \overline{S}_2, \overline{S}_3)$ is separated. We know 
  from~\cite{Le96} that for all $i\in\{1,2,3\}$ we have two hyperplanes (in this case affine lines) tangent to sets of the collection $(\overline{S}_j)_{j\neq i}$ and inducing opposite orientation on these. Such lines cannot meet $\overline{S}_i$ by separation property, so one of them, denoted $H_i$, is such that $\overline{S}_i \subset H_i^>$ and $\overline{S}_j \subset H_i^{\leq}$ for $j\neq i$. In particular, note that $\conv((S_j)_{j\neq i})\subset H_i^{\leq}$.
   For $i,j\in\{1,2,3\}$ and $j\neq i$, the hyperplane $H_i$ is not only tangent to $\overline{S}_j$ but also to $S_j$: indeed take a support $y^j_i$ of $H_i$ in $\overline{S}_j$, it arises as a convex combination $y^j_i=\sum_{k\neq i}\lambda_k x_k$ with $x_i\in S_k$ for $y^j_i\in \widehat{S_i}$. By $S_i\subset \overline{S}_i$, we derive for all $k\neq i$, $x_k\in H_k$ or $\lambda_k=0$, the latter being ruled out by separation. Hence, let us denote by $x^j_i$ a support of hyperplane $H_i$ in $S_j$. Note that once again from the separation of $(\overline{S}_1, \overline{S}_2, \overline{S}_3)$, two supports of a tangent line in two different colors cannot be equal.

  If $x\coloneqq(a, b)^{T}$ and $y\coloneqq(a', b')^{T}$ are two distinct vectors of $\R^2$, we denote $x\wedge y \coloneqq (ab'-a'b)^{-1}(b-b', a'-a)^{T}$, the usual cross-product of two vectors in $\mathbb{P}^2$. As is customary, $h_1 \coloneqq x_1^2 \wedge x_1^3$ (resp.\ $h_2 \coloneqq x_2^3 \wedge x_2^1$ and $h_3 \coloneqq x_3^1 \wedge x_3^2$) is a normal vector to $H_1$ (resp.\ $H_2$ and $H_3$), and $\dotp{h_i}{x} + 1 = 0$ is an equation defining $H_i$.
  Furthermore, the intersection of $H_1$ and $H_2$ is given by $s_3\coloneqq h_1\wedge h_2$, or using the triple product formula, by

  \begin{equation}\label{eq:decomposition_cross}
    s_3 =  h_1 \wedge (x_2^3 \wedge x_2^1) = \frac{(\dotp{h_1}{x_2^1}+1)\, x_2^3 - (\dotp{h_1}{x_2^3}+1)\, x_2^1}{(\dotp{h_1}{x_2^1}+1) \hskip2.7ex  - (\dotp{h_1}{x_2^3}+1) \hskip2.7ex } \enspace.
\end{equation}

Because $x_2^1\in \overline{S}_1 \subset H_1^>$ and $x_2^3\in \overline{S}_3 \subset H_1^{\leq}$, we have that $\dotp{h_1}{x_2^1}+1$ is nonzero and $(\dotp{h_1}{x_2^1}+1)(\dotp{h_1}{x_2^3}+1)\leq 0$. As a result of~\eqref{eq:decomposition_cross}, $s_3$ indeed exists and arises as a convex combination of $x_2^3$ and $x_2^1$, so $s_3\in \conv({S}_1\cup {S}_3)$. By writing $s_3 = (x_1^2\wedge x_1^3)\wedge h_2$ as in~\eqref{eq:decomposition_cross}, we show likewise that $s_3$ is a convex combination of $x_1^2$ and $x_1^3$, thus $s_3\in\conv({S}_2 \cup {S}_3)$. This finally entails that $s_3\in \overline{S}_3$ and therefore $s_3\in H_3^>$. It now suffices to define $s_1\coloneqq h_2\wedge h_3$ and $s_2 \coloneqq h_3 \wedge h_1$ in a similar way  and consider $y= (s_1+s_2+s_3)/3$. It is clear that $y\in\conv(S_1\cup S_2\cup S_3)$, and for all $i\in\{1,2,3\}$, $y\in H_i^>$, in particular $y\notin \conv((S_j)_{j\neq i})$. As a consequence, $y$ is a colorful vector for $S_1$, $S_2$ and $S_3$.
\end{proof}

To conclude, we point out that another interesting problem is the computational complexity of determining whether the colorful interior is empty or not, in the case where the sets $S_i$ are polytopes. Remark that as a consequence of~\Cref{prop:equivHatBar}, if Conjecture~\ref{conj} holds, then we can determine if $\arcenciel \Scal$ is empty in polynomial time using linear programming. Alternatively, the problem could be tackled by studying the complexity of separating a point from the colorful interior. This is tightly linked with the computation of the tangent hyperplanes of \Cref{th:colint}, for which the status of the complexity is not well understood. 

\bibliographystyle{splncs04}

\end{document}